\documentclass[reqno]{amsart}
\usepackage{anyfontsize}
\usepackage{graphicx,psfrag,amssymb,amsmath,amsthm, bbm}
\usepackage[cmtip,all]{xy}
\usepackage{latexsym}
\usepackage{color}
\usepackage{hyperref, pdfsync, comment}
\usepackage{eucal} 
\usepackage{nomencl}
\newtheorem{theorem}{Theorem}

\newtheorem{lemma}[theorem]{Lemma}
\newtheorem{proposition}[theorem]{Proposition}

\theoremstyle{definition}
\newtheorem{defn}{Definition}

\theoremstyle{remark}
\newtheorem{remark}{Remark}

\numberwithin{equation}{section}

\def\beq{\begin{equation}}
\def\eeq{\end{equation}}
\def\beqn{\begin{equation*}}
\def\eeqn{\end{equation*}}

\def\bh{\textbf{h}}

\def\bX{\mathbf{X}}

\def\bbf{\textbf{f}}

\def\bh{\mathbf{h}}

\def\bs{\mathbf{s}}

\def\b1{{\boldsymbol{1}}}

\def\cH{\mathcal{H}}

\def\cJ{\mathcal{J}}

\def\cO{\mathcal{O}}

\def\cT{\mathcal{T}}

\def\cW{\mathcal{W}}

\def\IE{{\mathbb E}}

\def\IR{{\mathbb R}}

\def\IZ{{\mathbb Z}}

\def\fF{\mathfrak{F}}
\def\fG{\mathfrak{G}}

\def\pQ{\partial Q}

\def\bX{\mathbf{X}}

\def\bh{\mathbf{h}}

\def\bs{\mathbf{s}}

\def\b1{{\boldsymbol{1}}}

\def\cH{\mathcal{H}}

\def\cJ{\mathcal{J}}

\def\cO{\mathcal{O}}

\def\cT{\mathcal{T}}

\def\cW{\mathcal{W}}

\def\IE{{\mathbb E}}

\def\IR{{\mathbb R}}

\def\IZ{{\mathbb Z}}

\def\fF{\mathfrak{F}}
\def\fG{\mathfrak{G}}

\def\eps{\epsilon}

\def\pQ{\partial Q}

\def\fF{\mathfrak{F}}

\begin{document}

\title[Almost Sure Invariance Principle]{Non-stationary Almost Sure Invariance Principle 
for Hyperbolic Systems with Singularities}

\date{}

\author{Jianyu Chen}
\address{Department of Mathematics \& Statistics\\ 
University of Massachusetts Amherst}
\email{jchen@math.umass.edu.}
\author{Yun Yang}
\address{Mathematics Department\\
The Graduate Center, City University of New York}
\email{yyang@gc.cuny.edu.}
\author{Hong-Kun Zhang}
\address{Department of Mathematics \& Statistics\\ 
University of Massachusetts Amherst}
\email{hongkun@math.umass.edu.}
\thanks{H.-K.~Zhang is partially supported by the NSF Career Award (DMS-1151762).}


\dedicatory{Dedicated to the memory of Nikolai Chernov.}

\maketitle

\begin{abstract}
We investigate a wide class of two-dimensional
hyperbolic systems with singularities,  
and prove the almost sure invariance principle (ASIP) 
for the random process generated by sequences of dynamically H\"older observables.
The observables could be unbounded, and the process may be non-stationary 
and need not have linearly growing variances.
Our results apply to Anosov diffeomorphisms,  
Sinai dispersing billiards and their perturbations. 
The random processes  
under consideration are related to 
the fluctuation of Lyapunov exponents, 
the shrinking target problem, etc.
\end{abstract}

\tableofcontents
\printindex
\section{Introduction}

Since the pioneering work by Sinai \cite{Sin70} on dispersing billiards, 
the dynamical structures and stochastic properties have been extensively studied
for chaotic billiards \cite{GO74, BS80, BSC90, BSC91, C99, C01, Ma04, SzV04, 
CZ05a, CZ05b, C06, BG06, CM07, C08, CZ08, BCD11, SYZ13},
and also for abstract hyperbolic systems with or without singularities 
\cite{Sin72, P77, KS86, P92, Sataev92, Y98, Y99,  Sarig02, DL08, CD09, CZ09,   DZ11, DZ13, DZ14}.
Among all the physical measures, the SRB measures - named after Sinai \cite{Sin72}, 
Ruelle \cite{Ru78} and Bowen \cite{Bow70, Bow75} -
are shown to display several levels of ergodic properties, including 
the decay rate of correlations,
the large deviation principles and the central limit theorem, etc.

In this paper, we shall focus on the almost sure invariance principle (ASIP)
for a wide class of uniformly hyperbolic systems with singularities, 
which preserve a mixing SRB measure.
The ASIP ensures that partial sum of a random process can be approximated by
a Brownian motion with almost surely negligible error.
More precisely, we say that 
a zero-mean random process $\bX=\{X_n\}_{n\ge 0}$ with finite second moments 
satisfies an ASIP\footnote{
We may need to extend the partial sum process $\{\sum_{k=0}^{n-1} X_k\}_{n\ge 0}$
on a richer probability space without changing its distribution.
In the rest of this paper, we always assume this technical operation when we mention ASIP.
} 
for an error exponent $\lambda\in [0, \frac12)$,
if there exists a Wiener process $W(\cdot)$ such that 
\beq\label{ASIP0}
\left|\sum_{k=0}^{n-1} X_k - W(\sigma_n^2) \right|=\cO(\sigma_n^{2\lambda}), \ \ \text{a.s.},
\eeq
where $\sigma_n^2=\IE\left( \sum_{k=0}^{n-1} X_k\right)^2$ 
is the variance of the $n$-th partial sum. 
In particular, if $\sigma_n^2$ grows linearly in $n$ such that 
$\sigma_n^2=n\sigma^2+\cO(1)$ for some $\sigma\in [0,\infty)$,
it follows from \eqref{ASIP0} that 
\beqn
\left|\sum_{k=0}^{n-1} X_k - \sigma W(n) \right|=\cO (n^{\lambda}), \ \ \text{a.s.}.
\eeqn
The ASIP implies many other limit laws from statistics, 
such as the almost sure central limit theorem, 
the law of the iterated logarithm 
and the weak invariance principle (see \cite{PhSt75} for the details). 

There has been a great deal of work on the ASIP in the probability theroy, see for instance
\cite{PhSt75, BP79, E86, ShaoLu87, Shao93, Cuny11, Wu13, CunyMe15}.
In the context of the stationary process generated by bounded H\"older observables
over smooth dynamical systems with singularities, 
the ASIP was first shown by Chernov \cite{C06} for Sinai dispersing billiards. 
Later, a scalar and a vector-valued ASIP were later proved by Melbourne and Nicol \cite{MN05, MN09} for the Young towers.
By a purely spectral method, Gou\"{e}zel \cite{G10} extended the ASIP for stationary processes for a wide class of systems without assuming Young tower structure.
Gou\"ezel \cite{G10} also provided a pure probabilistic condition, 
which was used by Stenlund \cite{Sten14} to show the ASIP for Sinai billiards with random scatterers.

The ultimate goal of our work is to establish ASIP 
for \emph{non-stationary} process generated by \emph{unbounded} observables, over a wide class of two-dimensional hyperbolic systems under the standard assumptions (\textbf{H1})-(\textbf{H5}) in Section~\ref{Sec:Assumptions}.  
Such class includes Anosov diffeomorphisms, 
Sinai dispersing billiards and their perturbations (See Section~\ref{sec: app} for more details).
Compared to existing results of ASIP for bounded stationary processes, 
our result is relatively new due to the following two features:
\begin{itemize}
\item[(1)] Low regularity: the question on how large classes of observables satisfy 
the central limit theorem or other limit laws had been raised 
by several researchers, see, e.g., a survey by Denker \cite{Den89}. 
Sometimes those classes are much larger than those of bounded H\"older continuous or bounded variation functions.
In this paper, we only assume dynamically H\"{o}lder continuity for observables, which could even be unbounded. 
A direct application is the fluctuation of Lyapunov exponents,
for which the log unstable Jacobian blows up near singularity in billiard systems.

\item[(2)] Non-stationarity: time-dependent processes 
arise from the dynamical Borel-Cantelli Lemma and the shrinking target problem (see e.g. \cite{HV95, CK01, Fa06}).
Recently, Haydn,  Nicol, T\"{o}r\"{o}k and  Vaient \cite{HNTV17} obtained ASIP for the shrinking target problem on a class of expanding maps. In analogy, under some mild conditions, 
we are able to apply our ASIP result  
to the shrinking target problem for two-dimensional hyperbolic systems with singularity.
\end{itemize}

The method we address here is rather transparent and efficient.
We first construct a natural family of $\sigma$-algebras that are characterized by the singularities,
and explore its exponentially $\alpha$-mixing property.
Extending the approach by Chernov \cite{C06} and 
applying a martingale version of ASIP by Shao \cite{Shao93}, 
we are then able to prove the ASIP for the random process
generated by a sequence of integrable dynamically H\"older observables.
A crucial assumption is that the process 
satisfies the Marcinkiewicz-Zygmund type inequalities given by \eqref{M-Z}.
We emphasize that those observables could be unbounded, 
and the growth of partial sum variances need not be linear.
Furthermore, 
the error exponent $\lambda$ in ASIP of the form \eqref{ASIP0} only depends on 
the constant $\kappa_2$ in \eqref{M-Z}.

This paper is organized as follows. In Section \ref{Sec: Assumptions and Results},
we introduce the standard assumptions for the uniformly hyperbolic systems 
with singularities, and state our main results on the ASIP and other limit laws.
We recall several useful theorems in probability theory in Section~\ref{sec: prelim}, 
and prove our main theorem on the  ASIP in Section \ref{sec:proof}.
In Section \ref{sec: app}, we summarize the validity of the ASIP for a wide class of 
uniformly hyperbolic billiards, and discuss two practical process related to the fluctuation of 
ergodic average and the shrinking target problem.

\section{Assumptions and Main Results}\label{Sec: Assumptions and Results}

\subsection{Assumptions}\label{Sec:Assumptions}

Let $T: M\to M$ be a piecewise $C^2$ diffeomorphism of 
a two-dimensional compact Riemannian manifold $M$ with singularities $S_1$,
that is, for each connected component $\Omega\subset M\backslash S_1$,
the map $T: \Omega \to T(\Omega)\subset M$ is a $C^2$ diffeormophism which
can be continuously extended to the closure of $\Omega$.
We denote by $S_{-1}:= T S_{1}$ the singularity of the inverse map $T^{-1}$.

Let $d(\cdot, \cdot)$ denote the distance in $M$ induced by the Riemannian metric.
For any smooth curve $W\subset M$, we denote by $|W|$ its length,
and by $m_W$ the Lebesgue measure on $W$ induced by the Riemannian metric restricted to $W$.

We now make several specific assumptions on the system $T: M\to M$. 
These assumptions are quite standard and have been made in many references  \cite{C99,CD09,CM,CZ09}.
\\

\noindent(\textbf{H1}) \textbf{Uniform hyperbolicity of $T$}. There exist two families of cones
$C^u_x$ (unstable) and $C^s_x$ (stable) in the tangent spaces
${\mathcal T}_x M$, for all $x\in M$, and there exists a constant $\Lambda>1$, with the following properties:
\begin{itemize}
\setlength\itemsep{0em}
\item[(1)] $D_x T (C^u_x)\subset C^u_{ T x}$ and $D_x T
    (C^s_x)\supset C^s_{ T x}$, wherever $D_x T $ exists.
\item[(2)] $\|D_x T(v)\|\geq \Lambda \|v\|$ for any
    $v\in C_x^u$, and $\|D_xT^{-1}(v)\|\geq
    \Lambda \|v\|$ for any $v\in C_x^s$. 
\item[(3)] These families of cones are  continuous on $M$
 and the angle between $C^u_x$ and $C^s_x$ is uniformly
 bounded away from zero.
 \end{itemize}

We say that a smooth curve $W\subset M$ is an \emph{unstable curve} for $T$ 
if at every point $x \in W$ the tangent line
$\cT_x W$ belongs in the unstable cone $C^u_x$.
Furthermore, a curve $W\subset  M$ is an \emph{unstable manifold} for $T$ if 
$T^{-n}(W)$ is an unstable curve for all $n \geq 0$.
We can define stable curves and stable manifolds in a similar fashion.\\

\noindent (\textbf{H2}) \textbf{Singularities.}  
The singularity set $S_{1}$ consists of a finite or countable union of smooth compact curves in $M$,
including the boundary $\partial M$.
We assume the following:
\begin{itemize}
\setlength\itemsep{0em}
\item[(1)] $\partial M$ is transversal to both stable and unstable cones.
\item[(2)] Every other smooth singularity curve in $S_1\setminus \partial M$ is a stable curve, and
every curve in $ S_1$ terminates either inside another curve of $ S_1$ or on $\partial M$. 
\item[(3)] There exist $C>0$ and $\beta_0\in (0,1)$ such that for any $x\in M\backslash S_1$,
\beqn
\|D_xT\|\le C d(x, S_1)^{-\beta_0}.
\eeqn
\end{itemize}
Similar assumptions are made for $S_{-1}$. We set
$
S_{\pm n}=\bigcup_{k=0}^{n-1}  T^{\mp k} S_{\pm 1}
$
for any $n\ge 1$, and it is clear that 
$S_{\pm n}$ be the singularity set of $T^{\pm n}$.
Furthermore, we denote 
$
S_{\pm\infty}=\bigcup_{n=0}^{\infty} S_{\pm n}. 
$

An unstable curve $W\subset M$ is said to be homogeneous  
if for any $n\ge 0$, 
$T^{-n} W$ is contained in a connected component of $M\setminus S_1$.
In other words, $W\cap S_{-\infty}=\emptyset$. 
Similarly, we can define homogeneous stable curves.

\begin{defn}\label{separation time}
For every $x, y\in M$, define $\bs_+(x,y)$, the 
forward separation time for $x$ and $y$, to be the smallest integer
$n\geq 0$ such that $x$ and $y$ belong to distinct elements of
$M\setminus S_n$.
Similarly we define the backward separation time $\bs_-(x,y)$.
\end{defn}

\noindent(\textbf{H3}) \textbf{Regularity of smooth unstable curves}. We assume that there is a $T$-invariant 
family $\cW^u_T$ of unstable curves such that 
\begin{enumerate}
\setlength\itemsep{0em}
\item[(1)]  \textbf{Bounded curvature} The  curvature of any $W\in \cW^{u}_T$ is uniformly bounded 
from above by a positive constant $B$.
 \item[(2)] \textbf{Distortion bounds.} There exist $\gamma_0\in (0,1)$ and  $C_T>0$ such
    that for any $W\in \cW^{u}_T$ and any $x, y \in W$, 
 \beqn\label{distor10}
      \left|\ln\cJ_W (x)-\ln \cJ_W (y)\right|   \leq C_{T}\, d (x, y)^{\gamma_0}, 
 \eeqn
where  $\cJ_W (x)=|D_x T|_{\cT_x W}|$ is  the Jacobian
of $T$ at $x$ along the unstable curve $W$.
\item[(3)] { \textbf{Absolute continuity.}}
 Let $W_1,W_2\in \cW^{u}_T$ be two unstable curves close to each other. Denote
 \begin{equation*}
 W_i'=\{x\in W_i\colon
W^s(x)\cap W_{3-i}\neq\emptyset\}, \hspace{.5cm} i=1,2.
 \end{equation*} The map
$\bh\colon W_1'\rightarrow W_2'$ defined by sliding along stable
manifolds
 is called the \textit{stable holonomy} map. 
 We assume that $\bh_*m_{W_1'}$ is absolutely continuous with respect to $m_{W_2'}$. 
 Furthermore, there exist $C_T>0$ and $\vartheta_0\in (0,1)$ such that  the Jacobian of $\bh$ satisfies
 \beqn\label{Jh}
 |\ln\cJ\bh(y)-\ln \cJ\bh(x)| \leq C_T \vartheta_0^{\bs_+(x,y)},
\ \ \ \text{for any}\ \ x, y\in W_1'.
\eeqn
\end{enumerate}

\noindent(\textbf{H4}) {\textbf{ SRB measure.}} 
The map $T$ preserves an SRB probability measure $\mu$, that is, the conditional measure of $\mu$ 
on each unstable manifold $W^u$ is absolutely continuous with respect to $m_{W^u}$.
We further assume that $\mu$ is strongly mixing. \\

\noindent(\textbf{H5}) {\textbf{ One-step expansion.}}
Given an unstable curve $W\subset M$, we denote $V_\alpha$
as the connected component in $TW$ with index $\alpha$ and $W_\alpha=T^{-1} V_\alpha$.
There is $q_0\in (0, 1]$ such that 
\beqn
  \liminf_{\delta\to 0}\
   \sup_{W\colon |W|<\delta}\sum_{\alpha}
   \left(\frac{|W|}{|V_{\alpha}|}\right)^{q_0} \frac{|W_\alpha|}{|W|}<1,
      \label{step1} 
\eeqn 
where the supremum is taken over all unstable curves $W$ in $M$.

\subsection{Statement of the main results}

The main result in this paper is to prove the almost sure invariance principle 
for the system $(M, T, \mu)$, which satisfies Assumptions \textbf{(H1)}-\textbf{(H5)}, 
with respect to the process
generated by a sequence of dynamically H\"older observables.
We first recall the definition of such functions.

\begin{defn} A measurable function $f: M\to \IR$ is said to be forward dynamically H\"older continuous if
there exists $\vartheta\in (0, 1)$ such that 
\beqn
|f|_{\vartheta}^+:=\sup\left\{\frac{|f(x)-f(y)|}{\vartheta^{\bs_+(x,y)}}:\ x\ne y\ \text{lie 
on a homogeneous unstable curve} \right\}<\infty,
\eeqn
where $\bs_{+}(\cdot, \cdot)$ is the forward separation time given by Definition \ref{separation time}.
The constant $\vartheta$ is called the dynamically H\"older exponent of $f$, and is usually denoted by $\vartheta_f$.
We denote the space of such functions by $\cH^+_\vartheta$, 
and set $\cH^+:=\cup_{\vartheta\in (0, 1)} \cH^+_\vartheta$. 

In a similar fashion, we define the space $\cH_\vartheta^-$ and $\cH^-$ of backward 
dynamically H\"older continuous functions. Also, we denote 
$\cH_\vartheta:=\cH^+_\vartheta\cap \cH^-_\vartheta$, and 
$\cH:=\cH^+\cap \cH^-$.
\end{defn}

\begin{remark}
Note that any H\"older continuous function is automatically 
dynamical H\"older continuous. 
However, a dynamically H\"older function can be only piecewise continuous, 
and it may not be bounded. 
\end{remark}

We need to assume certain integrability for the observables. 
Given an $L^s$-integrable function $f$ on $M$ for some $s\ge 1$, we denote
$\IE(f)=\int f d\mu$ and $\|f\|_{L^s}=\IE(|f|^s)^{1/s}$. 

For convenience, we shall use the following notations: 
given two sequence $a_n$ and $b_n$ of non-negative numbers,
we write $a_n=o(b_n)$ if $\lim_{n\to\infty} a_n/b_n=0$;
we write $a_n=\cO(b_n)$ or $a_n\ll b_n$ if $a_n\le Cb_n$ for some constant $C>0$,
which is independent of $n$;
and we denote $a_n\asymp b_n$ if $a_n\ll b_n$ and $a_n\gg b_n$.

We are now ready to state our main result.

\begin{theorem}\label{main}
Let $\bX_{\bbf}=\{X_n\}_{n\ge 0}:=\{f_n\circ T^n\}_{n\ge 0}$ be a random process generated by
a sequence $\bbf=\{f_n\}_{n\ge 0}$ of functions, which satisfies the following conditions: 
\begin{itemize}
\setlength\itemsep{0em}
\item[(1)] There are $\vartheta_\bbf\in (0, 1)$ and $\beta_\bbf\in [0, \infty)$ 
such that $f_n\in \cH_{\vartheta_\bbf}$ and
\beqn
|f_n|_{\vartheta_\bbf}^+ + |f_n|_{\vartheta_\bbf}^-  \ll n^{\beta_\bbf}.
\eeqn
\item[(2)] There is $p>4$ such that 
$f_n\in L^{p}$ with $\IE(f_n)=0$. Moreover, there are constants
$\kappa_p\ge \kappa_2>\frac{1}{4}$ such that 
\beq\label{M-Z}
\sigma_n:=\left\|\sum_{k=0}^{n-1} X_k \right\|_{L^2} \gg n^{\kappa_2}, \ \
\text{and}\ \ \
\sup_{m\ge 0} \left\|\sum_{k=m}^{m+n-1} X_k \right\|_{L^p} \ll n^{\kappa_p}.
\eeq
\end{itemize}
Then the process $\bX_{\bbf}$ satisfies an ASIP for any error exponent 
$\lambda\in \left( \max\{\frac14, \frac{1}{8\kappa_2}\}, \frac{1}{2}\right)$, that is,
there exists a Wiener process $W(\cdot)$ such that
\beq\label{ASIP1}
\left|\sum_{k=0}^{n-1} X_k - W(\sigma_n^2) \right|=\cO(\sigma_n^{2\lambda}), \  \ \text{a.s.}.
\eeq
\end{theorem}

\vskip.1in

\begin{remark}
It is well known that a zero-mean independent process
$\bX=\{X_n\}_{n\ge 0}$ with finite $s$-th moment (for some $s\ge 1$) 
satisfies the Marcinkiewicz-Zygmund inequalities, i.e.,
$
 \left\|\sum_{k=m}^{m+n-1} X_k \right\|_{L^s} \asymp 
 \left\| \left(\sum_{k=m}^{m+n-1}  X_k^2 \right)^{\frac12}\right\|_{L^s}.
$
Such type of inequalities were later generalized to martingale difference sequence, 
strongly mixing processes, etc (see e.g. \cite{MZ37, Yo80}).
We note that the term $\left\| \left(\sum_{k=m}^{m+n-1}  X_k^2 \right)^{\frac12}\right\|_{L^s}$ is of order $\sqrt n$
for stationary iid. processes.
Due to the dependence and non-stationarity in our setting, there is no a priori information on
$\left\| \left(\sum_{k=m}^{m+n-1}  X_k^2 \right)^{\frac12}\right\|_{L^s}$. 
To this end,  in terms of powers of $n$,
we directly impose the $2$nd moment lower bound and $p$-th moment upper bound 
in \eqref{M-Z}  for the partial sum $\sum_{k=m}^{m+n-1} X_k$.

\end{remark}

Condition (1) in Theorem~\ref{main} implies that every function $f_n$ is dynamically H\"older continuous
with a common exponent $\vartheta_\bbf$, while the dynamically H\"older semi-norms of $f_n$ are allowed to
grow in a polynomial rate.
Condition (2) implies that the growth rate of partial sum variances $\sigma_n^2$ 
is of order between $n^{2\kappa_2}$ and $n^{2\kappa_p}$. 
In particular, if $\kappa_2=\kappa_p=\frac12$, then
the growth is asymptotically linear, i.e., $\sigma_n^2\asymp n$. 

We notice that the error exponent $\lambda$ in \eqref{ASIP1} does not depend 
on the values of $\vartheta_\bbf$, $\beta_\bbf$ and $\kappa_p$,
and it can be chosen arbitrarily close to $\frac{1}{4}$ if $\kappa_2=\frac12$.
In the case when $p\in (2, 4]$ and $\kappa_2>\frac{1}{p}$, 
our result still holds with $\lambda\in \left( \max\{\frac14, \frac{1}{2p\kappa_2}\}, \frac{1}{2}\right)$,
but requires advanced moment inequalities in the proof of a technical lemma - Lemma~\ref{lem: est Rj}.
For simplicity, we just prove the case when $p>4$, which is sufficient for all of our applications.

\vskip.1in

Note that the ASIP is the strongest form - it implies many other limit laws, such as 
the weak invariance principle, 
the almost sure central limit theorem, 
and the law of iterated logarithm (see e.g. \cite{PhSt75} for their proofs and more details). 

\begin{theorem} 
Let $\bX_{\bbf}=\{X_n\}_{n\ge 0}:=\{f_n\circ T^n\}_{n\ge 0}$ be the random process
satisfying the assumptions in Theorem~\ref{main}.
We have the following limit laws:
\begin{itemize}
\setlength\itemsep{0em}
\item[(1)] Weak Invariance Principle: for any $t\in [0, 1]$, 
\beqn
\dfrac{1}{\sigma_n}\sum_{k=0}^{\lfloor nt\rfloor-1} f_k\circ T^k 
\xrightarrow{\ \ in\ distribution\ \ } W(t), \ \ \ \text{as}\ n\to \infty,
\eeqn
where $W(\cdot)$ is a Wiener process.
\item[(2)] Almost Sure Central Limit Theorem: 
we denote $S_n=\sum_{k=0}^{n-1} f_k\circ T^k$, and let $\delta(\cdot)$ be 
the Dirac measure on $\IR$, then for $\mu$-almost every $x\in M$,
\beqn
\dfrac{1}{\log \sigma_n^2} \sum_{k=1}^{n} \frac{1}{\sigma_k^2} \delta_{S_k(x)} 
\xrightarrow{\ \ in\ distribution\ \ } N(0, 1), 
\ \ \ \text{as}\ n\to \infty,
\eeqn
where $N(0, 1)$ is the standard normal distribution.
\item[(3)] Law of Iterated Logarithm: for $\mu$-almost every $x\in M$,
\beqn
\limsup_{n\to \infty} \dfrac{\sum_{k=0}^{n-1} f_k\circ T^k(x)}{\sqrt{2\sigma_n^2 \log\log \sigma_n^2}} =1.
\eeqn
\end{itemize}
\end{theorem}

\vskip.1in

\section{Preliminaries from Probability Theory}\label{sec: prelim}

In this section, we recall several useful theorems in the probability theory.
Let $(M, \mu)$ be a standard probability space. 

\begin{lemma}[Borel-Cantelli lemma]\label{lem: B-C}
If $\{E_n\}_{n\ge 1}$ is a sequence of events on $(M, \mu)$
such that $\sum_{n=1}^\infty \mu(E_n)<\infty$, 
then $\mu\left( \cap_{n=1}^\infty \cup_{k\ge n} E_k \right)=0.$
\end{lemma}

We introduce a special case of the results by Gal-Koksma (Theorem A1 in \cite{PhSt75}).

\begin{proposition}\label{Gal-Koksma}
Let $\{X_n\}_{n\ge 0}$ be a sequence of zero-mean random variables
with finite second moments. Suppose there is $\kappa>0$ such that   
for any $m\ge 0$,  $n\ge 1$,
\beqn
\IE\left(\sum_{k=m}^{m+n-1} X_k\right)^2\ll (m+n)^\kappa - m^\kappa,
\eeqn
then for any $\delta>0$,
\beqn
\sum_{k=0}^{n-1} X_k =o\left(n^{\frac{\kappa}{2} + \delta} \right), \ \ \text{a.s.}.
\eeqn
\end{proposition}

Let $\fF$ and $\fG$ be two $\sigma$-algebras on the space $(M, \mu)$.

\begin{defn} 
The $\alpha$-mixing coefficient between $\fF$ and $\fG$ is given by 
\beq\label{alpha coeff}
\alpha(\fF, \fG):=\sup_{A\in \fF} \sup_{B\in \fG} 
\left|\mu(A\cap B)-\mu(A)\mu(B) \right|.
\eeq
\end{defn}

Note that $\alpha(\fF, \fG)\le 2$. We have the following covariance inequality.

\begin{lemma}[Lemma 7.2.1 in \cite{PhSt75}]\label{lem: cov inequality} 
Let $s_1, s_2$ and $s_3$ be positive real numbers such that $1/s_1+1/s_2+1/s_3=1$.
For any $X\in L^{s_1}(M, \fF, \mu)$ and any $Y\in L^{s_2}(M, \fG, \mu)$, 
\beqn
\left|\IE(XY)-\IE(X)\IE(Y)\right|\le 10  \alpha(\fF, \fG)^{\frac{1}{s_3}} \|X\|_{L^{s_1}} \|Y\|_{L^{s_2}}.
\eeqn
\end{lemma}

\begin{defn}\label{def mds} 
$\{(\xi_j, \fG_j)\}_{j\ge 1}$ is called a martingale difference sequence if 
\begin{itemize}
\setlength\itemsep{0em}
\item[(1)] $\fG_j$ is an increasing sequence of $\sigma$-algebras on $(M, \mu)$;
\item[(2)] Each $\xi_j$ is $L^1$-integrable and $\fG_j$-measurable;
\item[(3)] $\IE(\xi_j|\fG_{j-1})=0$ for any $j\ge 2$.
\end{itemize}
\end{defn}

Here is a basic identity for martingale difference sequence $\{(\xi_j, \fG_j)\}_{j\ge 1}$: 
\beq\label{martingale diff identity}
\IE(X\xi_j)=0,
\eeq
for any $\fG_{j-1}$-measurable random variable $X$, as long as $X\xi_j\in L^1$ .

We shall need the following almost sure invariance principle by Shao \cite{Shao93}
for the martingale difference sequences (in the $L^4$-integrable case).

\begin{proposition}[\cite{Shao93}]\label{Shao ASIP} 
Let $\{(\xi_j, \fG_j)\}_{j\ge 1}$ be an $L^4$-integrable martingale difference sequence.
Put $b_r^2=\IE\left( \sum_{j=1}^r \xi_j \right)^2$. Assume that there exists a sequence 
$\{a_r\}_{r\ge 1}$ of non-decreasing positive numbers with $\lim_{r\to\infty} a_r=\infty$ such that
\begin{eqnarray}
& & \sum_{j=1}^r \left[ \IE\left(\xi_j^2\left| \fG_{j-1} \right. \right) - \IE\xi_j^2 \right] = o(a_r), \ a.s. \label{Shao Cond1} \\
& & \sum_{j=1}^\infty a_j^{-2} \IE|\xi_j|^{4} <\infty.  \label{Shao Cond2}
\end{eqnarray}
Then $\{\xi_j\}_{j\ge 1}$ satisfies an ASIP of the following form: 
there exists a Wiener process  $W(\cdot)$ such that 
\beqn
\left|\sum_{j=1}^r \xi_j - W(b_r^2)\right|
=o\left( \left( a_r \left( \left|\log(b_r^2/a_r)\right| + \log\log a_r \right)    \right)^{1/2} \right), \ a.s.
\eeqn
\end{proposition}

\section{Proof of Theorem~\ref{main}}\label{sec:proof}

We shall prove our main theorem as follows. 
Firstly, 
we construct a natural family $\fF$ 
of $\sigma$-algebras on $(M, \mu)$, and show that such family is exponentially $\alpha$-mixing.
Secondly, 
we introduce blocks and approximate the sequence $\bbf$ by 
conditional expectation over a special sub-family of $\fF$ on each block.
Furthermore,
we divide the partial sum of $\bX_\bbf$ into a major part $\sum_{j=1}^{r(n)-1} Y_j$
and other negligible parts.
Thirdly, 
we establish the martingale difference 
representation $\{\xi_j\}_{j\ge 1}$ for the process $\{Y_j\}_{j\ge 1}$,
and obtain several preliminary norm estimates.
Fourthly, 
we prove a technical lemma on Condition~\eqref{Shao Cond1} 
and show an ASIP for the martingale difference sequence $\{\xi_j\}_{j\ge 1}$.
Finally, 
we prove the ASIP for the original sequence $\bbf$.

\subsection{The strong mixing property}

We first recall the exponential decay of correlations for the system $ (M, T, \mu)$
for bounded dynamically H\"older observables, which was proven in \cite{CZ09} 
by using the coupling lemma (see e.g. \cite{CM, CD09}).

\begin{proposition}[\cite{CZ09}]\label{exp decay}
There exist $C_0>0$ and $\vartheta_0\in (0, 1)$ such that
for any pair of functions $f\in \cH_{\vartheta_f}^+\cap L^\infty(\mu)$ and $g\in \cH_{\vartheta_g}^-\cap L^\infty(\mu)$ 
and $n\ge 0$,
\beqn
\left| \IE(f\cdot g\circ T^n)  -\IE(f)\IE(g) \right| \le C_{f, g} \theta_{f, g}^n,
\eeqn
where $\theta_{f, g}=\max\{\vartheta_0, \vartheta_f^{1/4}, \vartheta_g^{1/4}\}<1$, and
\beqn
C_{f, g}=C_0\left(\|f\|_{L^\infty} \|g\|_{L^\infty} + \|f\|_{L^\infty} |g|_{\vartheta_g}^- + \|g\|_{L^\infty} |f|_{\vartheta_f}^+ \right).
\eeqn
\end{proposition}

We then introduce the following  natural family of $\sigma$-algebras for the system $T: M\to M$.
Recall that $S_{\pm n}$ is the singularity set of $T^{\pm n}$ for $n\ge 1$. 
Let $\xi_0:=\{M\}$ be the trivial partition of $M$,
and denote by $\xi_{\pm n}$ the partition of $M$ into connected components of 
$M\backslash T^{\mp(n-1)} S_{\pm 1}$ for $n\ge 1$.
Further, let 
$$
\xi_m^n:=\xi_m\vee \dots \vee \xi_n
$$
for all $-\infty\le m\le n\le \infty$.
By Assumption (\textbf{H2}), $\xi_0^\infty$ is the partition of $M$ into maximal unstable manifolds,
and $\xi_{-\infty}^0$ is that into maximal stable manifolds. 
Also, $\mu(\partial \xi_m^n)=0$ by Assumption (\textbf{H4}), where
$\partial \xi_m^n$ is the set of boundary curves for components in $\xi^n_m$.

Let $\fF_m^n$ be the Borel $\sigma$-algebra generated by the partition $\xi_m^n$.
Notice that 
$\fF_{-\infty}^\infty$ coincides with the $\sigma$-algebra of all measurable subsets in $M$.
We denote by $\fF:=\{\fF_m^n\}_{-\infty\le m\le n\le \infty}$ the family of those $\sigma$-algebras. 

\begin{proposition}\label{prop: alpha mixing}
 The family $\fF$ is $\alpha$-mixing with an exponential rate, i.e.,
there exist $C_0>0$ and $\vartheta_0\in (0, 1)$ (which are the same as in Proposition \ref{exp decay}) such that 
\beqn
\sup_{k\in \IZ}  \alpha(\fF_{-\infty}^k, \fF_{k+n}^\infty)\le C_0\vartheta_0^n,
\eeqn
where the definition of $\alpha(\cdot, \cdot)$ is given by \eqref{alpha coeff}.
\end{proposition}

\begin{proof} By the fact that $T^{-k}\xi_m^n=\xi_{m+k}^{n+k}$ and the invariance of $\mu$, it suffices to show that
\beqn
\alpha(\fF_{-\infty}^0, \fF_{n}^\infty)=\sup_{A\in \fF_{-\infty}^0} \sup_{B\in \fF_{n}^\infty} 
\left|\mu(A\cap B)-\mu(A)\mu(B) \right| \le C_0\vartheta_0^n.
\eeqn
Since $A\in \fF_{-\infty}^0$ is a union of some maximal stable manifolds, we have that $\b1_A\in \cH^-$
such that $\|\b1_A\|_{L^\infty}=1$ and $|\b1_A|_{\vartheta}^-=0$ for any $\vartheta\in (0, 1)$.
Similarly, $B\in \fF_{n}^\infty$ implies that $T^{-n}(B)\in \fF^\infty_0$ 
is a union of some maximal unstable manifolds,
and thus $\b1_{T^{-n}B}\in \cH^+$ such that $\|\b1_{T^{-n}B}\|_{L^\infty}=1$ and
$|\b1_{T^{-n}B}|_{\vartheta}^+=0$ for any $\vartheta\in (0, 1)$.
Therefore, by Proposition~\ref{exp decay}, for any $A\in \fF_{-\infty}^0$ and $B\in \fF_{n}^\infty$, 
\beqn
\left|\mu(A\cap B)-\mu(A)\mu(B) \right|=
\left| \IE(\b1_{T^{-n}B}\cdot \b1_A\circ T^n) -\IE(\b1_{T^{-n}B}) \IE(\b1_A ) \right|
\le C_0\vartheta_0^n.
\eeqn
This completes the proof of Proposition~\ref{prop: alpha mixing}. 
\end{proof}

\subsection{Blocks and approximations}

Let $\bbf=\{f_n\}_{n\ge 0}$ be a sequence of functions satisfying the assumptions of Theorem~\ref{main}. 

From now on, we fix a error exponent $\lambda\in \left(\max\{\frac14, \frac{1}{8\kappa_2}\}, \frac{1}{2}\right)$,
and choose a sufficiently small constant $\eps>0$ such that 

\beq\label{choose eps}
2\eps \kappa_p + \frac{1}{4-\eps} < 2\kappa_2 \lambda,
\ \ \text{and} \ \ \frac{\eps\kappa_p}{\kappa_2}<4\lambda-1.
\eeq

We partition the time interval $[0, \infty)$ into a sequence of consecutive blocks 
$\Delta_j=[\tau_j, \tau_{j +1})$ for $j\ge 1$, where  $\tau_j=\sum_{i=0}^{j-1} \lceil i^\eps\rceil$. 
Note that the block $\Delta_j$ is of length $\lceil j^\eps\rceil$, 
and $\tau_j\asymp j^{1+\eps}$.
For convenience, we set $\tau_0=-1$.

For any $k\in \Delta_j$,
we define the approximated function of $f_k$ by
\beq\label{g_k}
g_k=\IE\left(f_k\left|\fF^{\lceil 0.2 j^\eps\rceil}_{-\lceil 0.2 j^\eps\rceil}\right.\right).
\eeq
It is clear that $\IE(g_k)=0$. 
Since the separation times are adapted to the natural family $\fF$ of $\sigma$-algebras,
we have the following uniform $L^\infty$-bounds on the difference sequence $\{(f_k-g_k)\}_{k\ge 0}$.

\begin{lemma}\label{lem: convergence} 
$
\sup_{k\in \Delta_j} \|f_k-g_k\|_{L^\infty}\ll \vartheta_\bbf^{0.1 j^\eps}.
$
\end{lemma}

\begin{proof} Note that $k\le \tau_{j+1}\asymp j^{1+\eps}$ for any $k\in \Delta_j$.
For any measurable set $A\in \xi^{\lceil 0.2 j^\eps\rceil}_{-\lceil 0.2 j^\eps\rceil}$, 
and any two points $x, y\in A$, 
there is a point $z\in A$ such that $x$ and $z$ belong to one unstable curve, and 
$y$ and $z$ belong to one stable curve. It follows that 
$\bs_+(x, z)>\lceil 0.2 j^\eps\rceil$ and $\bs_-(y, z)>\lceil 0.2 j^\eps\rceil$, and thus 
by Condition (1) of Theorem~\ref{main},
{\allowdisplaybreaks
\begin{eqnarray*}
|f_k(x)-f_k(y)| &\le& |f_k(x)-f_k(z)|+|f_k(y)-f_k(z)| \\
&\le & (|f_k|^+_{\vartheta_\bbf}+|f_k|^-_{\vartheta_\bbf}) \ \vartheta_\bbf^{\lceil 0.2 j^\eps\rceil} \\
&\ll & k^{\beta_\bbf} \vartheta_\bbf^{\lceil 0.2 j^\eps\rceil} 
\le j^{(1+\eps)\beta_\bbf} \vartheta_\bbf^{\lceil 0.2 j^\eps\rceil} \ll \vartheta_\bbf^{0.1 j^\eps}.
\end{eqnarray*}
}
Hence for any $k\in \Delta_j$, we have
\begin{eqnarray*}
|f_k(x)-g_k(x)|
&=&\left|f_k(x)-\frac{1}{\mu(A)} \int_A f_k(y) d\mu(y)\right| \\
&\le &\frac{1}{\mu(A)} \int_A \left|f_k(x)-f_k(y) \right| d\mu(y) 
\ll \ \vartheta_\bbf^{0.1 j^\eps}.
\end{eqnarray*}
The proof of Lemma~\ref{lem: convergence} is complete. 
\end{proof}

For any $n\ge 0$, there is a unique $r(n)\ge 1$ such that $n\in \Delta_{r(n)}$.
Note that $r(n)\asymp n^{\frac{1}{1+\eps}}$.
We now decompose the partial sum of the process $\bX_\bbf=\{X_n\}_{n\ge 0}=
\{f_n\circ T^n\}_{n\ge 0}$
as follows:
\begin{eqnarray}\label{sum decomposition}
\sum_{k=0}^{n-1} X_k &=& 
\sum_{j=1}^{r(n)-1} \left(\sum_{k\in \Delta_j} g_k\circ T^k\right)  
+ \sum_{k=0}^{\tau_{r(n)}-1} (f_k-g_k)\circ T^k  
+\sum_{k=\tau_{r(n)}}^{n-1} X_k  \nonumber \\
&=: & \sum_{j=1}^{r(n)-1} Y_j + U_n+ V_n.
\end{eqnarray}
It turns out that the major contribution for ASIP is given by $\sum_{j=1}^{r(n)-1} Y_j$, 
while the rest terms are negligible. 

\begin{lemma}\label{lem: remainder} 
Let $U_n$ and $V_n$ be given by \eqref{sum decomposition}. Then
\begin{itemize}
\setlength\itemsep{0em}
\item[(i)] $\|U_n\|_{L^p}=\cO(1)$, and $|U_n|=\cO(1)$, a.s..
\item[(ii)]  $\|V_n\|_{L^p}=\cO\left(n^{\eps \kappa_p}\right)$,
and $|V_n|=\cO\left(n^{2\kappa_2\lambda}\right)$,  a.s..
\end{itemize}
\end{lemma}

\begin{proof} (i) Note that $U_n=\sum_{j=1}^{r(n)-1} \sum_{k\in \Delta_j}  (f_k-g_k)\circ T^k$. 
By Lemma~\ref{lem: convergence} and Minkowski's inequality, we have
\begin{eqnarray*}
\|U_n\|_{L^p} 
\le  \sum_{j=1}^\infty \sum_{k\in \Delta_j} \| (f_k-g_k)\circ T^k \|_{L^p} 
&\le & \sum_{j=1}^\infty \sum_{k\in \Delta_j} \| f_k-g_k \|_{L^p} \\
&\ll & \sum_{j=1}^\infty \lceil j^\eps \rceil \ \vartheta_\bbf^{0.1 j^\eps} < \infty.
\end{eqnarray*}
Moreover, $\sum_{j=1}^\infty \sum_{k\in \Delta_j} \| (f_k-g_k)\circ T^k \|_{L^p} <\infty$ implies that
$\sum_{j=1}^\infty \sum_{k\in \Delta_j} | (f_k-g_k)\circ T^k |<\infty$ a.s.,
and thus $|U_n|=\cO(1)$ a.s..

(ii) By \eqref{M-Z}, we obtain
\beqn
\|V_n\|_{L^p}=\left\|\sum_{k=\tau_{r(n)}}^{n-1} X_k \right\|_{L^p}
\ll \left(n-\tau_{r(n)}\right)^{\kappa_p} \ll \left(r(n)^\eps\right)^{\kappa_p}
\ll n^{\eps \kappa_p}.
\eeqn
Moreover, by  Markov's inequality and \eqref{choose eps}, 
\begin{eqnarray*}
\mu\{ |V_n|\ge n^{2 \kappa_2 \lambda} \}\le n^{-2p\kappa_2\lambda}\ \IE|V_n|^p
\ll n^{p(-2\kappa_2\lambda + \eps \kappa_p)} \ll n^{-\frac{p}{4-\eps}},
\end{eqnarray*}
and hence $\sum_{n=1}^\infty \mu\{ |V_n|\ge n^{2 \kappa_2 \lambda} \}<\infty$.
By the Borel-Cantelli lemma (Lemma~\ref{lem: B-C}), 
we get
$
\mu\left( \bigcap_{k=1}^\infty \bigcup_{n\ge k} \left\{ |V_n|\ge n^{2 \kappa_2 \lambda} \right\} \right)=0.
$
In other words,
$|V_n|\ll n^{2\kappa_2\lambda}$, a.s.. 
\end{proof}

\subsection{Martingale representation for $\{Y_j\}_{j\ge 1}$ }

In this subsection, we introduce a martingale representation for 
the random process $\{Y_j\}_{j\ge 1}$ as defined in \eqref{sum decomposition}. 
Such representation is given by Lemma 7.4.1 in \cite{PhSt75}, 
but has better norm estimates in our context. 

We first establish the following preliminary estimates for $Y_j$.

\begin{lemma}\label{lem: est Y} 
For any $j\ge 1$,  the random variable $Y_j$ is $\fF_{\tau_{j-1}}^{\tau_{j+2}}$-measurable
such that $\IE Y_j=0$ and $\|Y_j\|_{L^p}\ll j^{\eps \kappa_p}$.
Furthermore, $\|\sum_{j=1}^r Y_j\|_{L^2}\gg r^{\kappa_2}$.
\end{lemma}

\begin{proof}
By \eqref{g_k} and \eqref{sum decomposition}, we have for any $j\ge 1$,
\beqn
Y_j=\sum_{k\in \Delta_j} g_k\circ T^k
=\sum_{k\in \Delta_j} \IE\left(f_k\left|\fF^{\lceil 0.2 j^\eps\rceil}_{-\lceil 0.2 j^\eps\rceil}\right.\right)\circ T^k
=\sum_{k\in \Delta_j} \IE\left(f_k\circ T^k\left|\fF^{k+\lceil 0.2 j^\eps\rceil}_{k-\lceil 0.2 j^\eps\rceil}\right.\right)
\eeqn
is $\fF_{\tau_j-\lceil 0.2 j^\eps\rceil}^{\tau_{j+1}+\lceil 0.2 j^\eps\rceil}$-measurable,
and thus $\fF_{\tau_{j-1}}^{\tau_{j+2}}$-measurable. 
It is clear that $\IE Y_j=0$ since each $f_k$ is of zero mean.
Moreover, by Lemma~\ref{lem: convergence},
\beqn
\left\| Y_j-\sum_{k\in \Delta_j} X_k\right\|_{L^\infty} 
\le \sum_{k\in \Delta_j} \left\| f_k-g_k\right\|_{L^\infty}
\ll \lceil j^\eps\rceil  \vartheta_\bbf^{0.1 j^\eps} 
\le \sup_{j\ge 1} \lceil j^\eps\rceil  \vartheta_\bbf^{0.1 j^\eps} <\infty.
\eeqn
Therefore, by \eqref{M-Z}, 
\begin{eqnarray*}
\|Y_j\|_{L^p}
\le  \left\|\sum_{k\in \Delta_j} X_k\right\|_{L^p} 
+ \left\| Y_j-\sum_{k\in \Delta_j} X_k\right\|_{L^\infty}  
\ll   \lceil j^\eps\rceil^{\kappa_p} +\cO(1) \ll j^{\eps \kappa_p}. 
\end{eqnarray*}
Furthermore, 
\begin{eqnarray*}
\left\|\sum_{j=1}^r Y_j \right\|_{L^2}
&\ge & \left\|\sum_{k=0}^{\tau_{r+1}-1} X_k\right\|_{L^2}
-\sum_{j=1}^r \left\| Y_j-\sum_{k\in \Delta_j} X_k\right\|_{L^\infty} \\
&\gg & \tau_{r+1}^{\kappa_2} -  \sum_{j=1}^\infty \lceil j^\eps\rceil  \vartheta_\bbf^{0.1 j^\eps} \\
&\gg & (r+1)^{\kappa_2(1+\eps)}- \cO(1)\gg r^{\kappa_2}.
\end{eqnarray*}
\end{proof}

Now we denote $\fG_j$ the $\sigma$-algebra generated by $Y_1, Y_2, \dots, Y_j$,
and it is immediate from Lemma \ref{lem: est Y} that $\fG_j\subset \fF_{-1}^{\tau_{j+2}}$. 
We also set $\fG_0:=\{\emptyset, M\}$ to be the trivial $\sigma$-algebra.
 
\begin{lemma}\label{lem: martingale repn} 
For any $j\ge 1$, we set $\xi_j:=Y_j-u_j+u_{j+1}$, where $u_j\in L^{4}$ is given by 
\beq\label{def uj}
u_j:=\sum_{k=0}^\infty \IE\left(Y_{j+k}\left| \fG_{j-1}\right. \right).
\eeq
Then $\{(\xi_j, \fG_j)\}_{j\ge 1}$ is a martingale difference sequence.
Moreover, $\IE u_j=\IE \xi_j=0$ and 
\beqn
\|u_j\|_{L^{4}}\ll j^{\eps \kappa_p}, \ \ \text{and} \ \ \|\xi_j\|_{L^{4}}\ll j^{\eps \kappa_p}.
\eeqn
\end{lemma}

\begin{proof} We first show that each $u_j$, given by \eqref{def uj}, is well-defined as an $L^{4}$ function.
Denote for short $v_{jk}:=\IE\left(Y_{j+k}\left| \fG_{j-1}\right. \right)$, which is $\fG_{j-1}$-measurable.
Then
\begin{eqnarray}\label{Ev_jk}
\IE|v_{jk}|^{4}=\IE\left( v_{jk} \cdot v_{jk}^3 \right) 
= \IE\left( \IE\left(Y_{j+k}\left| \fG_{j-1}\right. \right) \cdot v_{jk}^3 \right) 
= \IE\left( Y_{j+k} \cdot v_{jk}^3 \right).
\end{eqnarray}
By Lemma~\ref{lem: est Y}, $Y_{j+k}$ is $\fF^{\tau_{j+k+2}}_{\tau_{j+k-1}}$-measurable, 
and also $\fG_{j-1}\subset \fF_{-1}^{\tau_{j+1}}$.
We choose $s_1=p$, $s_2=\frac{4}{3}$ and $s_3=\frac{4p}{p-4}$, 
and apply Lemma~\ref{lem: cov inequality} to the last term of \eqref{Ev_jk}, 
\begin{eqnarray*}
\IE|v_{jk}|^{4}
&\le & 10   \alpha(\fF^{\tau_{j+k+2}}_{\tau_{j+k-1}}, \fF_{-1}^{\tau_{j+1}})^{\frac{1}{s_3}} 
\left\|Y_{j+k} \right\|_{L^{s_1}} \left\| v_{jk}^3 \right\|_{L^{s_2}} \\
&=& 10   \alpha(\fF^{\tau_{j+k+2}}_{\tau_{j+k-1}}, \fF_{-1}^{\tau_{j+1}})^{\frac{1}{s_3}} 
\left\|Y_{j+k} \right\|_{L^{p}} \left[\IE|v_{jk}|^{4}\right]^{\frac{3}{4}}.
\end{eqnarray*}
Dividing $\left[\IE|v_{jk}|^{4}\right]^{\frac{3}{4}}$ on both sides, 
and then using Proposition~\ref{prop: alpha mixing} and Lemma~\ref{lem: est Y},
we have that for any $j\ge 1$, 
\begin{eqnarray*}
\|v_{jk}\|_{L^{4}}
&\le & 10  \alpha(\fF^{\infty}_{\tau_{j+k-1}}, \fF_{-\infty}^{\tau_{j+1}})^{\frac{1}{s_3}} \left\|Y_{j+k} \right\|_{L^{p}} \\
&\ll & 
\begin{cases}
\left( j+k \right)^{\eps \kappa_p}, \ & \ 0\le k<3, \\
\left( j+k \right)^{\eps \kappa_p} \vartheta_0^{\frac{\lceil (j+k-2)^\eps \rceil}{s_3}}, \ & \ k\ge 3,
\end{cases}\\
&\ll & 
\begin{cases}
j^{\eps \kappa_p}, \ & \ 0\le k<3, \\
\vartheta_0^{\frac{(k-2)^\eps}{2s_3}}, \ & \ k\ge 3.
\end{cases}
\end{eqnarray*}
Therefore, for any $j\ge 1$, 
\beqn
\sum_{k=0}^\infty \left\|v_{jk}\right\|_{L^{4}}=
\sum_{k=0}^2 \left\|v_{jk}\right\|_{L^{4}} + \sum_{k=3}^\infty \left\|v_{jk}\right\|_{L^{4}}
\ll 3 j^{\eps \kappa_p} + \sum_{k=3}^\infty \vartheta_0^{\frac{(k-2)^\eps}{2s_3}}
\ll j^{\eps \kappa_p},
\eeqn
which implies that $u_j=\sum_{k=0}^\infty v_{jk}$ is well-defined in $L^{4}$, 
and $\|u_j\|_{L^{4}}\ll j^{\eps \kappa_p}$.

By the formula $\xi_j:=Y_j-u_j+u_{j+1}$, 
it is easy to check that $\{(\xi_j, \fG_j)\}_{j\ge 1}$ is a martingale difference sequence (see Definition~\ref{def mds}).
Moreover,
\beqn
\|\xi_j\|_{L^{4}}\le \|Y_j\|_{L^{p}} +\|u_j\|_{L^{4}}+\|u_{j+1}\|_{L^{4}} \ll j^{\eps \kappa_p}. 
\eeqn
The proof of Lemma~\ref{lem: martingale repn} is complete. 
\end{proof}

The following lemma shows that $\sum_{j=1}^{r} Y_j$ is well approximated by
$\sum_{j=1}^{r} \xi_j$.

 \begin{lemma}\label{lem: xi remainder} We have the following estimates:
\beqn
\left\|\sum_{j=1}^{r} \left(Y_j-\xi_j \right) \right\|_{L^4}=\cO\left(r^{\eps\kappa_p}\right),
\ \ \text{and} \ \
\left|\sum_{j=1}^{r} \left(Y_j-\xi_j \right) \right|=\cO\left(r^{2\kappa_2\lambda}\right), \ a.s..
\eeqn
\end{lemma}

\begin{proof} By Lemma~\ref{lem: martingale repn}, we have 
$\sum_{j=1}^{r} \left(Y_j-\xi_j \right)=u_1-u_{r+1}$,
and thus
\beqn
\left\|\sum_{j=1}^{r} \left(Y_j-\xi_j \right) \right\|_{L^4}=\left\|u_1-u_{r+1} \right\|_{L^4}
\ll 1+(r+1)^{\eps\kappa_p} \ll r^{\eps\kappa_p}.
\eeqn
Moreover, by  Markov's inequality and \eqref{choose eps}, 
\begin{eqnarray*}
\mu\left\{ \left|\sum_{j=1}^{r} \left(Y_j-\xi_j \right) \right|\ge r^{2 \kappa_2 \lambda} \right\}
\le r^{-8\kappa_2\lambda}\ \IE\left| \sum_{j=1}^{r} \left(Y_j-\xi_j \right) \right|^4
&\ll & r^{4(-2\kappa_2\lambda + \eps \kappa_p)} \\
&\ll & r^{-\frac{4}{4-\eps}},
\end{eqnarray*}
and hence 
$\sum_{r=1}^\infty \mu\left\{ \left|\sum_{j=1}^{r} \left(Y_j-\xi_j \right) \right|\ge r^{2 \kappa_2 \lambda} \right\}<\infty$.
By the Borel-Cantelli lemma (Lemma~\ref{lem: B-C}), 
we get
$\left|\sum_{j=1}^{r} \left(Y_j-\xi_j \right) \right|\ll r^{2\kappa_2\lambda}$,\ a.s.. 
\end{proof}

According to Lemma~\ref{lem: remainder} and Lemma~\ref{lem: xi remainder}, 
we shall focus on proving ASIP for the process $\{\xi_j\}_{j\ge 1}$.

\subsection{ASIP for $\{\xi_j\}_{j\ge 1}$ }

In this subsection, we shall use Proposition~\ref{Shao ASIP} to
prove a version of ASIP for the martingale difference sequence $\{(\xi_j, \fG_j)\}_{j\ge 1}$.
We first need a technical lemma with the following almost sure estimate.

\begin{lemma}\label{lem: est Rj}
$\sum_{j=1}^r  \left[ \IE\left(\xi_j^2\left| \fG_{j-1} \right. \right) - \IE\xi_j^2 \right]= o(r^{4\kappa_2 \lambda})$, \ a.s.
\end{lemma}
\begin{proof}
We denote 
$R_j:=\IE\left(\xi_j^2\left| \fG_{j-1} \right. \right) - \IE\xi_j^2=\IE\left(\xi_j^2 -  \IE\xi_j^2\left| \fG_{j-1} \right. \right)$, 
and note that $R_j$ is $\fG_{j-1}$-measurable and $\IE R_j=0$.
Moreover, by Lemma~\ref{lem: martingale repn} and Jensen's inequality, 
\beqn
\|R_j\|_{L^2}\le \sqrt{\IE\left(\xi_j^2 -  \IE\xi_j^2\right)^2}\le \|\xi_j\|_{L^4}^2 \ll j^{2\eps \kappa_p}. 
\eeqn 
If for any $m\ge 1$ and any $r\ge 1$, 
\beq\label{est Rj}
\IE\left( \sum_{j=m}^{m+r-1}  R_j \right)^2 \ll (m+r)^{1+8\eps \kappa_p} - m^{1+8\eps \kappa_p} ,
\eeq
then Lemma~\ref{lem: est Rj} immediately follows from \eqref{choose eps} 
and Gal-Koksma (Proposition~\ref{Gal-Koksma}).
In the rest of the proof, we shall prove \eqref{est Rj}. Using that $\IE R_j=0$ and $\IE R_j^2\ge 0$, 
we first notice that 
{\allowdisplaybreaks
\begin{eqnarray*}
\IE\left( \sum_{j=m}^{m+r-1}  R_j \right)^2 
&\le & 2\sum_{j=m}^{m+r-1} \sum_{k=0}^{m+r-1-j} \IE(R_j R_{j+k}) \\
&= & 2\sum_{j=m}^{m+r-1} \sum_{k=0}^{m+r-1-j} 
\IE\left(R_j \IE\left(\xi_{j+k}^2 -  \IE\xi_{j+k}^2\left| \fG_{j+k-1} \right. \right) \right)\\
&=& 2\sum_{j=m}^{m+r-1} \sum_{k=0}^{m+r-1-j} \IE\left(R_j \left(\xi_{j+k}^2 -  \IE\xi_{j+k}^2 \right) \right) \\
&=& 2\sum_{j=m}^{m+r-1} \IE\left(R_j \sum_{k=0}^{m+r-1-j} \xi_{j+k}^2 \right) \\
&=& 2\sum_{j=m}^{m+r-1} \IE\left(R_j \left(\sum_{k=0}^{m+r-1-j} \xi_{j+k}  \right)^2\right) \\
& & - 4\sum_{j=m}^{m+r-1} \mathop{\sum\sum}_{0\le k<\ell\le m+r-1-j} \IE\left(R_j \xi_{j+k} \xi_{j+\ell}\right) \\
&=& 2\sum_{j=m}^{m+r-1} \IE\left(R_j \left(\sum_{k=0}^{m+r-1-j} \xi_{j+k}  \right)^2\right).
\end{eqnarray*} 
}In the last step, we use \eqref{martingale diff identity} to conclude that 
$\IE\left(R_j \xi_{j+k} \xi_{j+\ell}\right)=0$ if $k<\ell$.
By Lemma~\ref{lem: martingale repn}, we further obtain 
{\allowdisplaybreaks
\begin{eqnarray*}
& & \IE\left( \sum_{j=m}^{m+r-1}  R_j \right)^2 \\
&\le & 2\sum_{j=m}^{m+r-1} \IE\left(R_j \ \left[\sum_{k=0}^{m+r-1-j} Y_{j+k}   + \left(u_{m+r-1} - u_j\right) \right]^2\right) \\
&\le & 2\sum_{j=m}^{m+r-1} \left\{ \sum_{k=0}^{m+r-1-j}  \IE\left(R_j Y_{j+k}^2 \right) \right. + 2  \sum_{k=0}^{m+r-1-j} \sum_{\ell=1}^{m+r-1-j-k} \IE\left(R_j Y_{j+k} Y_{j+k+\ell} \right) \\
& & \ \ \ \ \ \ \ \ \ \ \ \ + 2  \sum_{k=0}^{m+r-1-j}  \IE\left(R_j Y_{j+k} u_{m+r-1} \right) 
 - 2   \sum_{k=0}^{m+r-1-j}  \IE\left(R_j Y_{j+k} u_j\right) \Bigg\}  \\
& & +2\sum_{j=m}^{m+r-1} \IE\left(R_j \left(u_{m+r-1} - u_j\right)^2 \right)  \\
&=:& 2\sum_{j=m}^{m+r-1} \left( I_1 + I_2 + I_3 + I_4  \right) + 2 I_5.
\end{eqnarray*}
}
To prove \eqref{est Rj}, it suffices to show  that 
\beqn
|I_i|\ll j^{8\eps \kappa_p}, \ \text{for}\ i=1, 2, 3, 4, \ \ \text{and}\ \ 
|I_5|\ll (m+r)^{1+8\eps \kappa_p}-m^{1+8\eps \kappa_p}.
\eeqn

For $I_1$: Recall that $\|R_j\|_{L^2}\ll j^{2\eps\kappa_p}$, and $R_j$ is $\fG_{j-1}$- and thus
$\fF_{-1}^{\tau_{j+1}}$-measurable.
By Lemma~\ref{lem: est Y},
$\|Y_{j+k}^2\|_{L^{p/2}}\le \|Y_{j+k}\|_{L^p}^2\ll (j+k)^{2\eps \kappa_p}$,
and $Y_{j+k}^2$ is $\fF^{\tau_{j+k+2}}_{\tau_{j+k-1}}$- and thus $\fF^{\infty}_{\tau_{j+k-1}}$-measurable.
Applying Lemma~\ref{lem: cov inequality} and Proposition~\ref{prop: alpha mixing}, we 
take $s=\frac{2p}{p-4}$ and get
{\allowdisplaybreaks
\begin{eqnarray*}
\left| I_1\right| &\le & \sum_{k=0}^{m+r-1-j}  \left| \IE\left(R_j Y_{j+k}^2 \right)\right| \\
&\le & \sum_{k=0}^{m+r-1-j} 10   \alpha(\fF_{-1}^{\tau_{j+1}}, \fF^{\infty}_{\tau_{j+k-1}})^{\frac{1}{s}} 
\left\|R_j \right\|_{L^{2}} \left\|Y_{j+k}^2 \right\|_{L^{p/2}} \\
&\ll &  \sum_{k=0}^2 j^{2\eps \kappa_p }(j+k)^{2\eps \kappa_p} 
+ \sum_{k=3}^{m+r-1-j} \vartheta_0^{\frac{\lceil (j+k-2)^\eps \rceil}{s}} j^{2\eps\kappa_p} (j+k)^{2\eps \kappa_p} \\
&\ll & j^{4\eps \kappa_p }\left[\cO(1)+ \sum_{k=3}^{m+r-1-j} \vartheta_0^{\frac{\lceil (k-2)^\eps \rceil}{s}}(1+k)^{2\eps \kappa_p} \right]
\ll j^{8\eps \kappa_p }.
\end{eqnarray*}
}

For $I_2$: we split the double sum into the cases $k\ge \ell$ and $k<\ell$, that is,
\beqn
|I_2|\le 2\sum_{k=0}^{m+r-1-j} \sum_{1\le \ell\le k} \left|\IE\left(R_j (Y_{j+k} Y_{j+k+\ell}) \right) \right|
+2\sum_{\ell=1}^{m+r-1-j} \sum_{0\le k<\ell} \left|\IE\left((R_j Y_{j+k} ) Y_{j+k+\ell} \right) \right|.
\eeqn
In the first summation, we note that $\ell\le k$ and
\beqn
\|Y_{j+k} Y_{j+k+\ell}\|_{L^{p/2}}\le \|Y_{j+k}\|_{L^p} \|Y_{j+k+\ell}\|_{L^p} 
\ll (j+k)^{\eps \kappa_p}(j+k+\ell)^{\eps \kappa_p} \ll (j+2k)^{2\eps \kappa_p}.
\eeqn
Applying Lemma~\ref{lem: cov inequality} and Proposition~\ref{prop: alpha mixing}, we 
take $s=\frac{2p}{p-4}$ and get
{\allowdisplaybreaks
\begin{eqnarray*}
& &\sum_{k=0}^{m+r-1-j} \sum_{1\le \ell\le k} \left|\IE\left(R_j (Y_{j+k} Y_{j+k+\ell}) \right) \right|  \\
&\le & \sum_{k=0}^{m+r-1-j} \sum_{1\le \ell\le k} 10   \alpha(\fF_{-1}^{\tau_{j+1}}, \fF^{\infty}_{\tau_{j+k-1}})^{\frac{1}{s}} 
\left\|R_j \right\|_{L^{2}} \left\|Y_{j+k} Y_{j+k+\ell} \right\|_{L^{p/2}} \\
&\ll &  \sum_{k=0}^2 k j^{2\eps \kappa_p }(j+2k)^{2\eps \kappa_p} 
+ \sum_{k=3}^{m+r-1-j} k \vartheta_0^{\frac{\lceil (j+k-2)^\eps \rceil}{s}} j^{2\eps\kappa_p} (j+2k)^{2\eps \kappa_p} \\
&\ll & j^{4\eps \kappa_p }\left[\cO(1)+ \sum_{k=3}^{m+r-1-j} \vartheta_0^{\frac{\lceil (k-2)^\eps \rceil}{s}}(1+2k)^{1+2\eps \kappa_p} \right]
\ll j^{8\eps \kappa_p }.
\end{eqnarray*}
}
In the second summation, we note that $k<\ell$ and hence
\beqn
\|R_jY_{j+k} \|_{L^{4/3}}\le \|R_j\|_{L^2}^{4/3} \|Y_{j+k}\|_{L^4}^{4/3}
\ll \left[j^{2\eps \kappa_p}(j+k)^{\eps \kappa_p}\right]^{4/3} \ll (j+k)^{4\eps \kappa_p}\le (j+\ell)^{4\eps \kappa_p}.
\eeqn
Also, $\|Y_{j+k+\ell} \|_{L^p}\ll (j+k+\ell)^{\eps\kappa_p}\le (j+2\ell)^{\eps \kappa_p}$.
Applying Lemma~\ref{lem: cov inequality} and Proposition~\ref{prop: alpha mixing}, we 
take $s'=\frac{4p}{p-4}$ and get
{\allowdisplaybreaks
\begin{eqnarray*}
& &\sum_{\ell=1}^{m+r-1-j} \sum_{0\le k<\ell} \left|\IE\left((R_j Y_{j+k} ) Y_{j+k+\ell} \right) \right|  \\
&\le & \sum_{\ell=1}^{m+r-1-j} \sum_{0\le k< \ell} 10   \alpha(\fF_{-1}^{\tau_{j+k+1}}, \fF^{\infty}_{\tau_{j+k+\ell-1}})^{\frac{1}{s'}} 
\left\|R_j Y_{j+k} \right\|_{L^{4/3}} \left\|Y_{j+k+\ell} \right\|_{L^{p}} \\
&\ll &  \sum_{\ell=1}^2 \ell (j+\ell)^{4\eps \kappa_p }(j+2\ell)^{\eps \kappa_p} 
+ \sum_{\ell=3}^{m+r-1-j} \ell \vartheta_0^{\frac{\lceil (j+k+\ell-2)^\eps \rceil}{s'}} 
(j+\ell)^{4\eps \kappa_p }(j+2\ell)^{\eps \kappa_p} \\
&\ll & j^{5\eps \kappa_p }\left[\cO(1)
+ \sum_{\ell=3}^{m+r-1-j} \vartheta_0^{\frac{\lceil (\ell-2)^\eps \rceil}{s}}(1+2\ell)^{1+5\eps \kappa_p} \right]
\ll j^{8\eps \kappa_p }.
\end{eqnarray*}
}
Therefore, $|I_2| \ll j^{8\eps \kappa_p}$. 

For $I_3$: by the definition of $u_j$ in \eqref{def uj}, we rewrite 
{\allowdisplaybreaks
\begin{eqnarray*}
\sum_{k=0}^{m+r-1-j}  \IE\left(R_j Y_{j+k} u_{m+r-1} \right)
&=& \sum_{k=0}^{m+r-1-j}  \sum_{\ell=0}^\infty \IE\left(R_j Y_{j+k} \IE\left( Y_{m+r-1+\ell} | \fG_{m+r-2} \right)\right) \\
&=& \sum_{k=0}^{m+r-1-j}  \sum_{\ell=0}^\infty \IE\left(R_j Y_{j+k}  Y_{m+r-1+\ell} \right) \\
&=& \sum_{k=0}^{m+r-1-j}  \sum_{\ell=m+r-1-j-k}^\infty \IE\left(R_j Y_{j+k}  Y_{j+k+\ell} \right)
\end{eqnarray*}
}
We can split $I_3$ into the cases $k\ge \ell$ and $k<\ell$, and obtain $|I_3|\ll j^{8\eps \kappa_p}$ by 
applying similar estimates for $I_2$.

For $I_4$: Note that 
$\|R_j u_j\|_{L^{4/3}}\le \|R_j\|_{L^2}^{4/3} \|u_{j}\|_{L^4}^{4/3}
\ll \left[j^{2\eps \kappa_p} j^{\eps \kappa_p}\right]^{4/3} = j^{4\eps \kappa_p}$.
Applying Lemma~\ref{lem: cov inequality} and Proposition~\ref{prop: alpha mixing}, we 
take $s'=\frac{4p}{p-4}$ and get
{\allowdisplaybreaks
\begin{eqnarray*}
|I_4| &\le & 2\sum_{k=0}^{m+r-1-j} \left| \IE\left(R_j Y_{j+k} u_j\right)\right| \\
&\le & 2 \sum_{k=0}^{m+r-1-j} 10   \alpha(\fF_{-1}^{\tau_{j+1}}, \fF^{\infty}_{\tau_{j+k-1}})^{\frac{1}{s'}} 
\left\|R_j u_j \right\|_{L^{4/3}} \left\|Y_{j+k} \right\|_{L^{p}} \\
&\ll &  \sum_{k=0}^2 j^{4\eps \kappa_p }(j+k)^{\eps \kappa_p} 
+ \sum_{k=3}^{m+r-1-j}  \vartheta_0^{\frac{\lceil (j+k-2)^\eps \rceil}{s'}} 
j^{4\eps \kappa_p }(j+k)^{\eps \kappa_p} \\
&\ll & j^{5\eps \kappa_p }\left[\cO(1)+ \sum_{k=3}^{m+r-1-j} 
\vartheta_0^{\frac{\lceil (k-2)^\eps \rceil}{s'}}(1+k)^{5\eps \kappa_p} \right]
\ll j^{8\eps \kappa_p }.
\end{eqnarray*}
}

For $I_5$: by Cauchy-Schwartz inequality,
{\allowdisplaybreaks
\begin{eqnarray*}
|I_5| 
\le \sum_{j=m}^{m+r-1} \left| \IE\left(R_j \left(u_{m+r-1} - u_j\right)^2 \right)\right|  
&\le & \sum_{j=m}^{m+r-1} \|R_j\|_{L^2}  \|u_{m+r-1}-u_j\|_{L^4}^2  \\
&\le & \sum_{j=m}^{m+r-1} \|R_j\|_{L^2}  \left(\|u_{m+r-1}\|_{L^4} + \|u_j\|_{L^4}\right)^2  \\
&\ll & \sum_{j=m}^{m+r-1} j^{2\eps \kappa_p} \left[  (m+r-1)^{\eps \kappa_p} + j^{\eps\kappa_p} \right]^2 \\
&\ll & (m+r)^{2\eps \kappa_p}  \sum_{j=m}^{m+r-1} j^{2\eps \kappa_p} \\
&\le & (m+r)^{2\eps \kappa_p}  \sum_{j=m}^{m+r-1} j^{6\eps \kappa_p} \\
&\le &  (m+r)^{1+8\eps \kappa_p} -m^{1+8\eps \kappa_p}.
\end{eqnarray*}
}
The proof of Lemma~\ref{lem: est Rj} is complete.
\end{proof}

We are now ready to show an ASIP for the sequence $\{\xi_j\}_{j\ge 1}$. 

\begin{lemma}\label{lem: ASIP xi} 
$\{\xi_j\}_{j\ge 1}$ satisfies an ASIP as follows: 
put $b_r^2=\IE\left( \sum_{j=1}^r \xi_j \right)^2$.
There exists a Wiener process $W(\cdot)$ such that 
\beqn
\left|\sum_{j=1}^r \xi_j - W(b_r^2)\right|
=o\left( r^{2\kappa_2\lambda(1+\eps)} \right), \ a.s..
\eeqn
\end{lemma}

\begin{proof} We directly apply Proposition~\ref{Shao ASIP} to the  
$L^4$-integrable martingale difference sequence $\{(\xi_j, \fG_j)\}_{j\ge 1}$. 
Set $a_r=r^{4\kappa_2\lambda}$, then 
Condition~\eqref{Shao Cond1} holds by Lemma~\ref{lem: est Rj}.
Condition~\eqref{Shao Cond2} also holds, since by Lemma~\ref{lem: martingale repn} and
\eqref{choose eps}, we have
\begin{eqnarray*}
\sum_{j=1}^\infty a_j^{-2} \IE|\xi_j|^{4}\ll \sum_{j=1}^\infty j^{-8\kappa_2\lambda} j^{4\eps \kappa_p}
\le \sum_{j=1}^\infty j^{-\frac{4}{4-\eps}}<\infty.
\end{eqnarray*}
On the one hand, by \eqref{choose eps}, $\eps\kappa_p<\kappa_2(4\lambda-1)<\kappa_2$, and thus
\beqn
b_r=\left\|\sum_{j=1}^r \xi_j \right\|_{L^2}\le \sum_{j=1}^r \left\|\xi_j \right\|_{L^2}\ll 
\sum_{j=1}^r j^{\eps \kappa_p}\ll r^{1+\eps\kappa_p}\ll r^{1+\kappa_2}.
\eeqn
On the other hand, by Lemma~\ref{lem: est Y} and Lemma~\ref{lem: xi remainder},
\beqn
b_r=\left\|\sum_{j=1}^r \xi_j \right\|_{L^2}\ge 
\left\|\sum_{j=1}^r Y_j \right\|_{L^2} - \left\|\sum_{j=1}^r (Y_j-\xi_j) \right\|_{L^4}
\gg r^{\kappa_2} - \cO(r^{\eps \kappa_p}) \gg r^{\kappa_2}.
\eeqn
Therefore, 
$r^{2\kappa_2(1-2\lambda)}\ll b_r^2/a_r\ll r^{2\kappa_2(1-2\lambda)+2}$,
and hence $\left|\log(b_r^2/a_r)\right|\ll r^{4\kappa_2\lambda\eps}$. 
It is obvious that $\log\log a_r\ll r^{4\kappa_2\lambda\eps}$ as well.
By Proposition~\ref{Shao ASIP}, we have 
{\allowdisplaybreaks
\begin{eqnarray*}
\left|\sum_{j=1}^r \xi_j - W(b_r^2)\right|
&=& o\left( \left( a_r \left( \left|\log(b_r^2/a_r)\right| + \log\log a_r \right)    \right)^{1/2} \right), \ a.s. \\
&\le & o\left( r^{2\kappa_2\lambda(1+\eps)} \right), \ a.s..
\end{eqnarray*}
}
\end{proof}

\subsection{ASIP for $X_\bbf$ }
Finally, we prove Theorem~\ref{main} - the ASIP for the random process
$X_\bbf=\{X_n\}_{n\ge 0}=\{f_n\circ T^n\}_{n\ge 0}$. By the previous subsections,
we can now write
\beq\label{sum decomposition2}
\sum_{k=0}^{n-1} X_k = \sum_{j=1}^{r(n)-1} \xi_j +
\sum_{j=1}^{r(n)-1} (Y_j-\xi_j) + U_n + V_n.
\eeq
We first compare the variances $\sigma_n^2=\IE\left(\sum_{k=0}^{n-1} X_k\right)^2$
and $b_{r(n)-1}^2=\IE\left(\sum_{j=1}^{r(n)-1} \xi_j\right)^2$. 

\begin{lemma}\label{lem: compare var}
$\left| \sigma_n - b_{r(n)-1}\right|\ll n^{\eps \kappa_p}$. 
As a result, for any Wiener process $W(\cdot)$, 
\beqn
\left|W(\sigma_n^2)-W(b_{r(n)-1}^2)\right| = \cO\left(\sigma_n^{2\lambda}\right), \ a.s.
\eeqn
\end{lemma}
\begin{proof} By \eqref{sum decomposition2}, Lemma~\ref{lem: remainder} and Lemma~\ref{lem: xi remainder},
{\allowdisplaybreaks
\begin{eqnarray*}
\left| \sigma_n - b_{r(n)-1} \right| 
&\le & \left\| \sum_{j=1}^{r(n)-1} (Y_j-\xi_j) + U_n + V_n \right\|_{L^2} \\
&\le & \left\| \sum_{j=1}^{r(n)-1} (Y_j-\xi_j) \right\|_{L^4} + \|U_n\|_{L^p} + \|V_n \|_{L^p} \\
&=& \cO\left(\left(r(n)-1\right)^{\eps \kappa_p}\right) + \cO(1) + \cO\left(n^{\eps \kappa_p}\right)\ll n^{\eps \kappa_p}.
\end{eqnarray*}
}    
In the last step, we use the fact that $r(n)\asymp n^{\frac{1}{1+\eps}}\ll n$.
By \eqref{M-Z} and \eqref{choose eps}, 
{\allowdisplaybreaks
\begin{eqnarray*}
\left|\sigma_n^2- b_{r(n)-1}^2\right|
&\le & \left| \sigma_n - b_{r(n)-1}\right|\left( 2\sigma_n + \left| \sigma_n - b_{r(n)-1}\right| \right) \\
&\ll & n^{\eps \kappa_p} \left( 2\sigma_n + n^{\eps \kappa_p} \right) \ll \sigma_n^{\eps \kappa_p/\kappa_2+1}. 
\end{eqnarray*}
}
For any Wiener process $W(\cdot)$, the random variables $Z_n:=W(\sigma_n^2)-W(b_{r(n)-1}^2)$
follows the normal distribution $N\left(0, \left|\sigma_n^2- b_{r(n)-1}^2\right| \right)$.
By \eqref{choose eps}, we can choose a sufficiently large 
$
s>\frac{4\max\{1, 1/\kappa_2\}}{4\lambda-1-\eps \kappa_p/\kappa_2}>4.
$
Then by Markov's inequality and Jensen's inequality, we have 
{\allowdisplaybreaks
\begin{eqnarray*}
\mu\{|Z_n|\ge \sigma_n^{2\lambda}\}\le \sigma_n^{-2\lambda s} \IE|Z_n|^s
\le \sigma_n^{-2\lambda s} \left(\IE|Z_n|^2\right)^{s/2}
&\ll & \sigma_n^{\frac{s}{2}\left[ \eps \kappa_p/\kappa_2+1 -4\lambda\right]} \\
&\ll & \sigma_n^{-2/\kappa_2}\ll n^{-2},
\end{eqnarray*}
}
which implies that $\sum_{n=1}^\infty \mu\{|Z_n|\ge \sigma_n^{2\lambda}\}<\infty$. 
By the Borel-Cantelli lemma (Lemma~\ref{lem: B-C}), 
we get
$|Z_n|\ll \sigma_n^{2\lambda}$, a.s.. 
\end{proof}

We are now ready to prove our main theorem.

\begin{proof}[Proof of Theorem~\ref{main}]
First, by \eqref{sum decomposition2}, Lemma~\ref{lem: remainder} and Lemma~\ref{lem: xi remainder},
we have 
{\allowdisplaybreaks
\begin{eqnarray*}
\left| \sum_{k=0}^{n-1} X_k - \sum_{j=1}^{r(n)-1} \xi_j  \right| 
&\le & \left| \sum_{j=1}^{r(n)-1} (Y_j-\xi_j) \right| + |U_n| + |V_n | \\
&=& \cO\left(\left(r(n)-1\right)^{2 \kappa_2\lambda}\right) + \cO(1) + \cO\left(n^{2 \kappa_2\lambda}\right), \ a.s. \\
&=& \cO\left(n^{2\kappa_2\lambda}\right) = \cO\left(\sigma_n^{2\lambda}\right), \ a.s.
\end{eqnarray*}
}
By Lemma~\ref{lem: ASIP xi} and Lemma~\ref{lem: compare var}, 
there exists a Wiener process $W(\cdot)$ such that 
{\allowdisplaybreaks
\begin{eqnarray*}
\left|\sum_{k=0}^{n-1} X_k - W(\sigma_n^2)\right|
&\le & \left|\sum_{k=0}^{n-1} X_k -  \sum_{j=1}^{r(n)-1} \xi_j \right| +
\left|\sum_{j=1}^{r(n)-1} \xi_j - W(b_{r(n)-1}^2)\right|  \\
& & + \left|W(\sigma_n^2)-W(b_{r(n)-1}^2)\right| \\
&=&  \cO\left(\sigma_n^{2\lambda}\right) + 
o\left( \left(r(n)-1\right)^{2\kappa_2\lambda(1+\eps)} \right) +\cO\left(\sigma_n^{2\lambda}\right) 
= \cO\left(\sigma_n^{2\lambda}\right), \ a.s..
\end{eqnarray*}
}Here we use the fact that $r(n)\asymp n^{\frac{1}{1+\eps}}$ and hence
$\left(r(n)-1\right)^{2\kappa_2\lambda(1+\eps)}\asymp n^{2\kappa_2 \lambda}\ll \sigma_n^{2\lambda}$.
This completes the proof of Theorem \ref{main}.
\end{proof}

\section{Applications to Random Processes for Concrete Systems }\label{sec: app}

\subsection{Concrete hyperbolic systems}\label{sec:concrete}

Our main result applies to a large class of two-dimensional uniformly hyperbolic systems, 
including Anosov diffeomorphisms\footnote{
By adding the boundaries of the finite Markov partition,
a topological mixing $C^2$ Anosov diffeomorphism satisfies our Assumptions \textbf{(H1)-(H5)}.
}
and chaotic billiards.
We shall focus on the Sinai dispersing billiards and their conservative perturbations.
Since such models were studied in \cite{CZ09, DZ11, DZ13}, 
we only remind some basic facts here.

We first recall standard definitions, see \cite{BSC90,BSC91,C99}. 
A  two-dimensional billiard is a dynamical system where a point moves 
freely at the unit speed in a domain $Q\subset \mathbb{R}^2$  and bounces off its
boundary $\pQ$ by the laws of elastic reflection. 
A billiard is dispersing if 
$\pQ$ is a finite union of mutually disjoint $C^3$-smooth curves
with strictly positive curvature.
Four broad classes of perturbations of the dispersing billiards were considered in \cite{DZ11, DZ13}:
\begin{itemize}
  \item[(a)]  Tables with shifted, rotated or deformed scatterers;
  \item[(b)]  Billiards under small external forces which bend trajectories during flight;
  \item[(c)]  Billiards with kicks or twists at reflections, including slips along the disk;
  \item[(d)]  Random perturbations comprised of maps with uniform properties (including
  any of the above classes, or a combination of them).
\end{itemize}

We treat all the above systems in a universal coordinate system. 
More precisely, let $ M=\pQ\times [-\pi/2,\pi/2]$ be the collision space, which
is a standard cross-section of the billiard flow. The canonical coordinate
in $M$ is denoted by $(r, \varphi)$, where $r$ is the arc length parameter
on $\pQ$ and $\varphi\in [-\pi/2,\pi/2]$ is the angle of reflection.
The collision map $T: M \to M$ takes an outward unit vector $(r, \varphi)$
at $\partial Q$ to the outward unit vector after the next collision, 
and the singularity of $T$ is caused by the tangential collisions, that is,
$S_1=\partial M\cup T^{-1}(\partial M)$.

It was proven in \cite{CZ09, DZ11, DZ13} that all the above collision map $T: M\to M$
preserves a mixing SRB measure $\mu$, and the systems $(M, T, \mu)$ 
satisfy the assumptions \textbf{(H1)-(H5)} in Section~\ref{Sec:Assumptions}. 
Therefore, under conditions in Theorem \ref{main}, 
the ASIP holds for the non-stationary process generated by unbounded observables
over those systems.

\subsection{Practical random process} 
In this subsection, we discuss some practical processes 
over the concrete systems in Section~\ref{sec:concrete}. 

\subsubsection{Fluctuation of Lyapunov exponents}

By Birkhoff's ergodic theorem, Pesin entropy formula 
and the mixing property of the system $(M, T, \mu)$, 
we have 
\beqn
\lim_{n\to\infty} \frac{1}{n} \log \left|D^u_x T^n \right| = \int \log |D^u_x T| d\mu = h_\mu(T),
\eeqn
where  $|D^u_x T^n|$ is the Jacobian of $T^n$ at $x$ along the unstable direction,
and $h_\mu(T)$ is the metric entropy of the SRB measure $\mu$.
We would like to study the fluctuation of the convergence for the ergodic sum given by 
\beqn
\log \left|D^u_x T^n \right| - nh_\mu(T)=\sum_{k=0}^{n-1} \left[\log |D^u_x T| - h_\mu(T)\right] \circ T^k.
\eeqn
Unlike in Anosov systems, 
the log unstable Jacobian function $x\mapsto \log |D^u_x T|$ is unbounded in billiard systems. 
Recall that $M$ is the phase space of a billiard system with coordinates
$x=(r, \varphi)$, then $\log |D^u_x T|\asymp -\log \cos\varphi$
blows up near the singularities $\{\varphi=\pm\frac{\pi}{2}\}$.
Nevertheless, $\log |D^u_x T|$ is 
dynamically H\"older continuous by Assumption (\textbf{H3}), 
and it belongs to $L^p$ for any $p\ge 1$, as 
\beqn
\int \left| \log |D^u_x T| \right|^p d\mu \asymp  \iint \left| \log \cos\varphi\right|^p \cos \varphi dr d\varphi <\infty.
\eeqn

More generally, it follows from Theorem~\ref{main} that  
an ASIP holds for the ergodic sum of a dynamically H\"older observable.
Here we only assume higher order integrability rather than boundedness for the observable.

\begin{theorem}\label{thm: stationary}
Suppose that $f\in \cH\cap L^p$ for some $p>4$, such that 
$\IE f=0$ and the first moment of its auto-correlations is finite, i.e.,
\beq\label{cor 1m}
\sum_{n=0}^\infty n\left|\IE(f \cdot f\circ T^n)\right| < \infty.
\eeq
Then the stationary process $\bX_f:=\{f\circ T^n\}_{n\ge 0}$ 
satisfies an ASIP for any error exponent $\lambda\in \left(\frac14, \frac12 \right)$,
that is, there is a Wiener process $W(\cdot)$ such that 
\beq\label{ASIP12}
\left|\sum_{k=0}^{n-1} f \circ T^k - \sigma_f W(n) \right|=\cO(n^{\lambda}), \  \ \text{a.s.}.
\eeq
where $\sigma_f^2$ is given by the Green-Kubo formula, i.e., 
\beq\label{sigma_f}
\sigma^2_f:=\sum_{n=-\infty}^\infty \IE\left(f\cdot f\circ T^n\right) \in [0, \infty).
\eeq
\end{theorem}

\begin{proof} 
First, we note that the series in \eqref{sigma_f} converges absolutely 
by Condition \eqref{cor 1m}. 
By direct computation, we have
{\allowdisplaybreaks
\begin{eqnarray*}
\sigma_n^2=\IE\left( \sum_{k=0}^{n-1} f \circ T^k\right)^2 
&= & n \sigma_f^2 -  \sum_{|k|>n} n \ \IE\left(f\cdot f\circ T^k\right)  -2 \sum_{k=1}^{n-1} k\ \IE\left(f\cdot f\circ T^k\right) \\
&=& n  \sigma_f^2 +\cO\left(1 \right).
\end{eqnarray*}
}Therefore, $\sigma_f^2=\lim_{n\to\infty} \sigma_n^2/n \in [0, \infty)$. 
If $\sigma_f=0$, then $\sigma_n^2$ is uniformly bounded.
In such case, it is well known that $f$ is a coboundary, i.e.,
there exists an $L^2$ function $g: M\to \IR$ such that $f=g-g\circ T$ (see e.g. Theorem 18.2.2 in \cite{IL71}), 
and thus \eqref{ASIP12} is automatic for any error exponent $\lambda>0$. 

We now focus on the case when $\sigma_f>0$. 
Condition (1) in Theorem \ref{main} automatically holds since $f\in \cH$.
For Condition (2), we have $\sigma_n\asymp \sqrt n$ since  $0<\sigma_f<\infty$, that is, $\kappa_2=\frac12$.
Also, by stationarity and Minkowski's inequality,
\beqn
\sup_{m\ge 0} \left\|\sum_{k=m}^{m+n-1} f \circ T^k \right\|_{L^p} =
\left\|\sum_{k=0}^{n-1} f \circ T^k \right\|_{L^p} \le n \|f\|_{L^p} \ll n.
\eeqn
In other words, $\kappa_p=1$. By Theorem \ref{main}, we obtain the ASIP for any $\lambda\in (\frac14, \frac12)$, 
that is, there exists a Wiener process $W(\cdot)$ such that 
\beqn
\left|\sum_{k=0}^{n-1} f \circ T^k - W(\sigma_n^2) \right|=\cO(\sigma_n^{2\lambda})=\cO\left(n^\lambda\right), \  \ \text{a.s.}.
\eeqn
Note that $Z_n:=W(\sigma_n^2)-\sigma_f W(n)$ follows the normal distribution 
$N\left(0, \left|\sigma_n^2- n\sigma_f^2\right|\right)$,
and recall that $\left|\sigma_n^2- n\sigma_f^2\right| =\cO(1)$. 
Then by Jensen's inequality,
\beqn
\mu\{|Z_n|\ge n^{\frac14}\}\le n^{-2} \IE|Z_n|^8
\le n^{-2} \left(\IE|Z_n|^2\right)^{4}
\ll n^{-2},
\eeqn
which implies that $\sum_{n=1}^\infty \mu\{|Z_n|\ge n^{\frac14}\}<\infty$. 
By the Borel-Cantelli lemma (Lemma~\ref{lem: B-C}), 
we get
$|Z_n|\ll n^{\frac14}$, a.s.. Therefore, 
\beqn
\left|\sum_{k=0}^{n-1} f \circ T^k - \sigma_f W(n) \right| \le
\left|\sum_{k=0}^{n-1} f \circ T^k - W(\sigma_n^2) \right| +\left| Z_n\right| 
=\cO(n^\lambda), \ a.s..
\eeqn
\end{proof}

\begin{remark}
The stationary ASIP in the special case when $p=\infty$ had been shown by Chernov \cite{C06}
and many other authors. In this case, Condition \eqref{cor 1m} holds due to  
the exponential decay of correlations for bounded dynamically H\"older observables.

Moreover, we can relax Condition \eqref{cor 1m} to sub-linear first moment of auto-correlations, i.e.,
$\sum_{n=0}^\infty n\left|\IE(f \cdot f\circ T^n)\right| \ll n^\eta$ for some $\eta<1$, then 
by a slight modification in the proof, we can show that the ASIP in \eqref{ASIP12}
holds for any error exponent $\lambda\in \left(\max\{\frac14, \frac{\eta}{2}\}, \frac12 \right)$.
\end{remark}

Now we can directly apply Theorem~\ref{thm: stationary} to study the fluctuations of Lyapunov exponents
in generic billiard systems for which Markov sieves exist (See Corollary 1.8 and Theorem 7.2 in \cite{C95} for more details).
For such generic billiards, Condition \eqref{cor 1m} holds for $f=\log |D^u_x T| - h_\mu(T)$ and 
a broader class of observables. 
Therefore, by Theorem~\ref{thm: stationary}, for any $\lambda\in (\frac14, \frac12)$, there is a Wiener process $W(\cdot)$ 
such that  
\beqn
\left|\log \left|D^u_x T^n \right| - nh_\mu(T) - \sigma_f W(n) \right|=\cO(n^{\lambda}), \  \ \text{a.s.}.
\eeqn

\vskip.1in

\subsubsection{Shrinking target problem}

Let $\{A_n\}_{n\ge 0}$ be a sequence of nested Borel subsets of $M$, i.e.,  
$A_n\supset A_{n+1}$ for any $n\ge 0$. 
Given $x\in M$, 
we can study the absolute frequency that the trajectory of $x$ hits the shrinking targets $A_n$.
More precisely, for any $n\ge 1$, we denote
\beq\label{def Nn}
N_n(x)=\#\{k\in [0, n):\ T^k x\in A_k\}=\sum_{k=0}^{n-1} \b1_{A_k}\circ T^k(x).
\eeq
Note that $\IE N_n=\sum_{k=0}^{n-1} \mu(A_k)$ by the invariance of $\mu$ under $T$. 
We say that the sequence $\{A_n\}_{n\ge 0}$ is dynamically Borel-Cantelli if 
$\lim_{n\to\infty} N_n=\infty$, a.s..

Similar to a recent result by Haydn, Nicol, T\"or\"ok and Vaienti in \cite{HNTV17} (see 
Theorem 5.1 therein),
we obtain the following ASIP for the frequency process $N_n$ of the shrinking target problem.

\begin{theorem}\label{thm: ASIP target}
Let $\{A_n\}_{n\ge 0}$ be a sequence of nested Borel subsets such that 
\begin{itemize}
\item[(i)] There is $\beta\in [0, \infty)$ such that 
$\b1_{A_n}\in \cH_{0.5}$ and
$
|\b1_{A_n}|_{0.5}^+ + |\b1_{A_n}|_{0.5}^-  \ll n^{\beta};
$
\item[(ii)] There is $\gamma\in (0, \frac34)$ such that $\mu(A_n)\gg n^{-\gamma}$.
Moreover, $\mu(A_n)=o(\frac{1}{\log n})$.
\end{itemize}
Then the process $N_n$ (as defined in \eqref{def Nn})
satisfies the ASIP for any error exponent 
$\lambda\in \left( \max\{\frac14, \frac{1}{8(1-\gamma)}\}, \frac{1}{2}\right)$, that is,
there exists a Wiener process $W(\cdot)$ such that
\beq\label{ASIP14}
\left|N_n - \IE N_n   -W(\sigma_n^2) \right|=\cO(\sigma_n^{2\lambda}), \  \ \text{a.s.},
\eeq
where $\sigma_n^2=\IE N_n^2 -\left( \IE N_n \right)^2$.
\end{theorem}

\begin{remark} Here is a particular choice of the sequence $\{A_n\}_{n\ge 0}$ such that 
Condition (i) in Theorem~\ref{thm: ASIP target} holds: let 
$A_n$ be an open subset with boundaries 
in the singular set $S_{-w(n)}\cup S_{w(n)}$,
where $w(n)$ is an sequence of positive integers such that 
$w(n)\ll \log_2 n$. Then each $\b1_{A_n}\in \cH_{0.5}$ and
$
|\b1_{A_n}|_{0.5}^+ + |\b1_{A_n}|_{0.5}^- \le 2^{w(n)} \ll n^{\beta}
$
for some $\beta>0$.
\end{remark}

\begin{proof}[Proof of Theorem~\ref{thm: ASIP target}] 
Without loss of generality, we may assume that $\mu(A_0)\le \frac12$. 
We take $f_n=\b1_{A_n}-\mu(A_n)$, then $N_n - \IE N_n=\sum_{k=0}^{n-1} f_k\circ T^k$. 
It follows from Condition (i) that $\{f_n\}_{n\ge 0}$ satisfies Condition (1) in Theorem \ref{main}.
For Condition (2), the second moment inequality in \eqref{M-Z} is automatic since for any $p\ge 1$,
any $m\ge 0$ and any $n\ge 1$, 
\beqn
\left\|\sum_{k=m}^{m+n-1} f_k\circ T^k \right\|_{L^p} \le \sum_{k=m}^{m+n-1} \left\|f_k \right\|_{L^\infty} \le 2n.
\eeqn
That is, $\kappa_p=1$. 

It remains to show the first moment inequality in \eqref{M-Z}. 
We follow the arguments of Lemma 2.4 in \cite{HNVZ13}.
First, we claim the following long term iterations:
there exists $c>0$ such that for any $k\ge 0$,
\beq\label{long iterations}
\sum_{\ell>k+c\log(k+1)} \left| \IE(f_k\circ T^k \cdot f_\ell \circ T^\ell) \right| \ll (k+1)^{-2}.
\eeq
Indeed, by Proposition~\ref{exp decay}, we take $\theta:=\max\{\vartheta_0, 2^{-1/4}\}<1$
and $c>\frac{3+\beta}{-\log \theta}$.
Then together by Condition (i), for any $0\le k< \ell$ and such that $\ell-k>c\log(k+1)$,  
{\allowdisplaybreaks
\begin{eqnarray*}
\left| \IE(f_k\circ T^k \cdot f_\ell \circ T^\ell) \right|
&\ll & C_0 \left(4+2k^\beta +2\ell^\beta \right) \theta^{\ell-k} \\
&\le & C_0 \left(4+2k^\beta + 2^{1+\beta} \left[k^\beta + (\ell-k)^\beta  \right] \right)\theta^{\ell-k} \\
&\ll & \cO(1) + k^\beta \theta^{\ell-k} + (\ell-k)^\beta \theta^{\ell-k} \\
&\ll & \cO(1) + k^\beta (k+1)^{-3-\beta} +\cO(1)\ll (k+1)^{-3},
\end{eqnarray*}
}which immediately implies \eqref{long iterations}. Now we have 
{\allowdisplaybreaks
\begin{eqnarray*}
\sigma_n^2=\IE\left( \sum_{k=0}^{n-1} f_k\circ T^k \right)^2
&=& \sum_{k=0}^{n-1} \IE(f_k^2) 
 +2\sum_{k=0}^{n-1} \sum\limits_{\substack{k< \ell<n, \\ \ell \le k+c\log(k+1)}} \IE(f_k\circ T^k \cdot f_\ell \circ T^\ell) \\
& & +2 \sum_{k=0}^{n-1} \sum\limits_{\substack{k< \ell<n, \\ \ell>k+ c\log(k+1)}} \IE(f_k\circ T^k \cdot f_\ell \circ T^\ell) \\
&=& \sum_{k=0}^{n-1} \left( \mu(A_k) - \mu(A_k)^2 \right)  \\
& & + 2\sum_{k=0}^{n-1} \sum\limits_{\substack{k< \ell<n, \\ \ell \le k+ c\log(k+1)}} 
\left( \mu(A_k \cap T^{-(\ell-k)}A_\ell) - \mu(A_k) \mu(A_\ell) \right)  \\
& & + 2 \sum_{k=0}^{n-1} \cO\left((k+1)^{-2} \right) \\
&\ge & \sum_{k=0}^{n-1} \left( \mu(A_k) - \mu(A_k)^2 \right) 
- 2\sum_{k=0}^{n-1} \sum\limits_{\substack{k< \ell<n, \\ \ell \le k+c\log(k+1)}} \mu(A_k) \mu(A_\ell) +\cO(1) \\
&\ge & \frac12 \sum_{k=0}^{n-1} \mu(A_k) - 2c \sum_{k=0}^{n-1} \log(k+1) \mu(A_k)^2 +\cO(1) \\
&=& \frac12 \sum_{k=0}^{n-1} \mu(A_k) \left[1-4c \log(k+1) \mu(A_k)\right] +\cO(1),
\end{eqnarray*}
}where the last inequality uses the fact that $\mu(A_\ell)\le \mu(A_k)\le \mu(A_0)\le \frac12$. 
By Condition (ii), we further get
\beqn
\sigma_n^2\gg \sum_{k=0}^{n-1} \mu(A_k) (1-o(1))\gg n^{1-\gamma}. 
\eeqn
In other words, $\kappa_2=1-\gamma$. 
Applying our main theorem - Theorem~\ref{main}, we obtain the ASIP for $N_n$ given by 
\eqref{ASIP14}. 
\end{proof}

\section*{Acknowledgements}
%H.-K.~Zhang is partially supported by the NSF Career Award (DMS-1151762). 
The authors would like to thank Nicolai Haydn and Huyi Hu for helpful discussions and suggestions.

\end{document}